\documentclass[11pt]{amsart}

  \usepackage[colorlinks, pagebackref, linkcolor=red, citecolor=blue]{hyperref}
 \usepackage{amsmath,amssymb}
 \usepackage{color}
 \usepackage{mathrsfs}
 \usepackage{graphicx}
 \usepackage{lineno}

 \usepackage{calc}
\makeatletter
 
\newlength{\temp@wc@width}

\newlength{\temp@wc@height}

\newcommand{\widecheck}[1]{%
\setlength{\temp@wc@width}{\widthof{$#1$}}%
\setlength{\temp@wc@height}{\heightof{$#1$}}%
#1\hspace{-\temp@wc@width}%
\raisebox{\temp@wc@height+2pt}[\heightof{$\widehat{#1}$}]%
{\rotatebox[origin=c]{180}{\vbox to 0pt{\hbox{$\widehat{\hphantom{#1}}$}}}}%
}

\makeatother
 
 \numberwithin{equation}{section}
 \theoremstyle{plain}
 \newtheorem{theorem}{Theorem}[section]
 \newtheorem{lemma}[theorem]{Lemma}

 \newtheorem{corollary}[theorem]{Corollary}
 \newtheorem{proposition}[theorem]{Proposition}

   \theoremstyle{definition}
\newtheorem{definition} [theorem] {Definition}

 \newcommand{\zit}[1]{(\ref{#1})}

\def \n{\raisebox{-2pt}{\rule{1.3pt}{10pt}}}
 
\def\bsr{\operatorname{bsr}}
\def\tsr{\operatorname{tsr}}

\def\tsr{\operatorname{tsr}}
\def\usr{\operatorname{usr}}

\def\nor{\operatorname{nor}}
\def\aur{\operatorname{aur}}
\def\snr{\operatorname{snr}}

 \def\bs{\boldsymbol}

 \def\R{ \mathbb R}

 \def\D{{ \mathbb D}}
  \def\K{{ \mathbb K}}
 \def\C{{ \mathbb C}}

 \def\N{{ \mathbb N}}

 \def\e{\varepsilon}

 \def\bs{\boldsymbol}
 \def\dis{\displaystyle}

 \def\inter{\cap}
 \def\Inter{\bigcap }
 \def\ov{\overline}

 \def\ss{\subseteq}
 \def\emp{\emptyset}

 \def\buildrel#1_#2^#3{\mathrel{\mathop{\kern 0pt#1}\limits_{#2}^{#3}}}
 
\begin{document}

 \title [Stable ranks]{Unitary stable ranks and norm-one ranks}


 \author{Raymond Mortini}
  \address{
Universit\'{e} de Lorraine\\
 D\'{e}partement de Math\'{e}matiques et  
Institut \'Elie Cartan de Lorraine,  UMR 7502\\
 Ile du Saulcy\\
 F-57045 Metz, France} 
 \email{raymond.mortini@univ-lorraine.fr}


 
 \begin{abstract}
In the context of commutative $C^*$-algebras we solve    a problem related to a question
of M. Rieffel by showing that the all-units rank  and the norm-one  rank coincide
with the topological stable rank. We also introduce the notion of unitary $M$-stable rank for
an arbitrary commutative unital ring and compare it with the Bass stable rank.
In case of uniform algebras, a sufficient condition for norm-one reducibility is given.
 \end{abstract}

  \maketitle

 \centerline {\small\the\day.\the \month.\the\year} \medskip
 \section*{Introduction}
 
 Let $C$ be a $C^*$ algebra with identity. Given a pair $(a,b)$ of elements in $C$
 for which $aC+bC=C$, one can conclude from the work of Robertson 
 \cite{ro}   that there exist two units $u$ and $v$ in $C^{-1}$  with $ua+vb={\bs 1}$ if and only
 if  $A$ has dense invertible group. In that case there even  exists  a unitary element $u\in C$
 (that is an element satisfying $uu^*=u^*u={\bs 1}$)  such that $a+ub\in C^{-1}$.
  In his groundbreaking paper \cite[p. 307]{ri},
 Mark Rieffel posed the  problem whether there is
 an analogue for $C^*$-algebras $C$ with $\tsr C=n$.  This question was re-asked in \cite{bad}.
 We shall give a positive   answer to  weaker versions of 
 this question in context of the algebra $C(X,\K)$
 of $\K$-valued continuous functions on a compact Hausdorff space $X$, where $\K=\R$ or $\C$.
 To this end we give several possible ways of extending the definition of the unit-1-stable rank   
  (see  \cite{ccm}) from pairs $(a,b)$ to $(n+1)$-tuples. 
Some of them were briefly mentioned in \cite{bad}. Generally speaking, we replace ``unitary" elements in $C$ (which correspond to unimodular functions in $C(X,\K)$) either by
invertible elements (called units) or by norm-one elements. The original question by Rieffel
remains unanswered, though.

Let $R$ be  a commutative unital ring.  Then 
$$U_n(R)=\{\bs f=(f_1,\dots,f_n)\in R^n: \sum_{j=1}^n R f_j=R\}$$
is the set of invertible $n$-tuples. If $R$ carries a topology, then  
the {\it topological stable rank}, $\tsr R$,  of $R$ is the smallest integer $n$
 for which $U_n(R)$ is dense in
$R^n$ (or infinity if $U_n(R)$ is never dense). This concept was introduced by Rieffel \cite{ri}. 
It is well known that within the realm of commutative unital Banach algebras $A$ one has
$\bsr A\leq \tsr A$, where $\bsr A$ is the Bass stable rank of $A$. Recall that this item is
defined to be the smallest integer $n$
for which any $(\bs f, g)\in U_{n+1}(R)$ is reducible in the sense that there exists
$\bs x\in R^n$ such that $\bs f+\bs x\, g\in U_n(R)$.

 Let us recall the following easy fact, which was one of the motivations for dubbing these items
 ``stable ranks" (they satisfy certain stabilizing properties): 
 
 \begin{proposition}\label{nton+1red}
Let $A$ be a commutative unital algebra.
Suppose that $\bsr A=n$, $n<\infty$, and let $m\geq n$.
 Then every invertible $(m+1)$-tuple $(\bs f,g)\in A^{m+1}$ is reducible. 
 \end{proposition}

 As usual,  a $Q$-algebra is  a commutative unital  topological algebra over $\K$ for which the set $A^{-1}$ of units is open.  If, additionally,  inversion $x\to x^{-1}$ is a continuous operation on $A^{-1}$, then  we call $A$ a $cQ$-algebra.
The following interesting characterization of the topological stable rank
 (see \cite[p. 52]{bad}) is the key to our results.

 \begin{theorem}\label{badi}
 Let $A=(A,|\cdot|)$ be a normed $cQ$-algebra. For 
 $\bs a=(a_1,\dots,a_n)\in A^n$, let
 $ ||\bs a||=\sum_{j=1}^n|a_j|$ be a fixed norm on the product space.
  Then the following assertions are equivalent:
 \begin{enumerate}
\item[(1)] $\tsr A\leq n$;
\item[(2)] For every $(\bs a,g)\in U_{n+1}(A)$ there is $\bs v\in U_n(A)$ and $\bs y\in A^n$
 such that
\begin{enumerate}
\item [i)]  $||\bs v-\bs a ||<\e$,
\item [ii)]  $\bs v=\bs a+ \bs y\, g$.
\end{enumerate}
\end{enumerate}
 
\end{theorem}

\section{The unitary stable ranks}

We begin with two possible extensions of the definition of the unit-1-stable rank. 
Recall that   a commutative unital ring has the {\it unit-1-stable rank} if for every invertible 
pair $(a,b)\in U_2(R)$ there exist $u,v\in R^{-1}$ such that $au+bv={\bs 1}$.
  In that case one says that $(a,b)$ is {\it totally reducible}.

 \begin{definition}
 Let $R$ be a commutative unital ring. 
 \begin{enumerate}
\item[(1)] The {\it unitary $M$-stable rank} \footnote{in order to distinguish our stable rank here from
the one given in \cite{mvdk}, I added my initial M here},
 $\usr R$, of $R$ is the smallest integer $n$ such that for every
 $(\bs a, b)\in U_{n+1}(R)$ there is $\bs u\in U_n(R)$ such that 
 $\bs a +\bs u\, b\in U_n(R)$.
 If there exists no such $n$, then we put $\usr R=\infty$. 
\item[(2)] The {\it all-units rank}, $\aur R$, of $R$ is the smallest 
integer $n$ such that for every
 $(\bs a, b)\in U_{n+1}(R)$ there are $u_j\in R^{-1}$ such that   
 $\bs a +\bs u \,b\in U_n(R)$, where $\bs u=(u_1,\dots, u_n)$.
 If there exists no such $n$, then we put $\aur R=\infty$.
\end{enumerate}
\end{definition}
 
 Note that $\bsr R\leq \usr R\leq \aur R$ is a trivial estimate. Thus, if $\aur R=1$, then
 $\bsr R=\usr R=\aur R=1$, and this holds  
 if and only if $R$ has the unit-1-stable rank.

 \newpage

 \begin{theorem}
Let $R$ be a commutative unital ring. Then 
\begin{enumerate}
\item[(1)] The unitary $M$-stable rank has the stabilizing property; that is if $\usr R=n<\infty$, and if $m\geq n$ then, for any $(\bs f,g)\in U_{m+1}(R)$  there is $\bs u \in U_m(R)$ 
such that $\bs f+\bs u\, g\in U_m(R)$.
\item[(2)] $\bsr R\leq \usr R\leq \bsr R+1.$
\end{enumerate}
Both cases in (2) can occur.
\end{theorem}
$\bullet$ I don't know whether the all-units rank has the stabilizing property.
\begin{proof}
(1) We may assume that $m\geq n+1$. Let $(f_1,\dots, f_m, g)\in U_{m+1}(R)$.
Then $(f_1,\dots,f_n, f_{n+1}+g, \dots, f_m+g, g)\in U_{m+1}(A)$, too.
Hence, there is $(a_1,\dots, a_{m+1})\in R^{m+1}$
such that 
\begin{equation}\label{bez-lang}
\sum_{j=1}^n  a_jf_j+\bigl( \sum_{j=n+1}^m a_j(f_j+g)+a_{m+1}g\bigr)
={\bs 1}.
\end{equation}
Put $h:=\sum_{j=n+1}^m a_j(f_j+g)+a_{m+1}g$. Then 
$$(f_1,\dots,f_n, h)\in U_{n+1}(R).$$
Since $\bsr R\leq \usr R=n$, there exists \footnote{ Here we may use Proposition \ref{nton+1red}  or
directly the assumption $\usr R=n$.}
 $(x_1,\dots, x_n)\in A^n$ such that 
$$(f_1+ x_1 h,\dots, f_n+x_n h)\in U_n(A);$$
that is
$$\mbox{$\sum_{j=1}^n y_j(f_j+x_j h)={\bs 1}$ for some $(y_1,\dots,y_n)\in A^n$}.$$
We claim that 
$$(f_1+x_1a_{m+1}g, \dots, f_n+x_na_{m+1}g, f_{n+1}+g,\dots, f_m+g)\in U_m(A).$$
To show this,   note that
$h$ has the form  $h=r+a_{m+1}g$, where $r\in I_A(f_{n+1}+g,\dots, f_m+g)$.
Hence
\begin{eqnarray*}
{\bs 1}&=&\sum_{j=1}^n y_j(f_j+x_ja_{m+1}g)+ \sum_{j=1}^n y_j x_j r\\
&\in& I_A(f_1+x_1a_{m+1}g,\dots, f_n+x_na_{m+1}g, f_{n+1}+g,\dots, f_{m}+g).
\end{eqnarray*}
If we put $u_j= x_ja_{m+1}$ for $j=1,\dots, n$ and $u_j=\bs 1$ for $j=n+1,\dots, m$, then
we see that $\bs f+ \bs c\, g\in U_m(A)$, where $\bs c=(c_1,\dots, c_m)$. Moreover, 
$\bs c\in U_{m}(R)$, since at least one coordinate is $\bs 1$.

(2) Since the first inequality $\bsr R\leq \usr R$ is obvious,  it remains to show that $\usr R\leq \bsr R+1$. But this follows from the proof of part (1) by putting $m=n+1$, where $n=\bsr R$.\\

Since $\tsr C([0,1]),\C)=1$, we may approximate the solution $(x,y)$ to $xa+yb=1$
by an invertible pair $(u,v)$. Hence $ua+vb$ is invertible again. So 
$\usr C([0,1],\C)=1=\bsr  C([0,1],\C)$. By \cite{moru92}, $(z,f)$ is not totally reducible
for every $f\in A(\D)$ with $f(0)\not=0$. Hence $\usr A(\D)\geq 2$.  But $\bsr A(\D)=1$,
(\cite{jmw}). Hence $\usr A(\D)=2$.
 \end{proof}

 Here is a first relation of the unitary $M$-stable rank to the topological stable rank. 
  
 \begin{proposition}\label{usr-bsr}
  Let $A$ be a $Q$-algebra. Then $\bsr A\leq \usr A\leq \tsr A$. 
\end{proposition}

$\bullet$~ I don't know whether $\usr A\leq \aur A\leq \tsr A$ or $\usr A\leq \tsr A\leq\aur A$
always holds for normed $Q$-algebras.

\begin{proof}
The first inequality, $\bsr A\leq \usr A$ is trivial. Now suppose that $n:=\tsr A<\infty$. Let
$(\bs a,b)\in U_{n+1}(A)$.  Then there is $\bs x\in A^n$ and $y\in A$ such that
$\bs x\cdot\bs a + y b=\bs 1$. Since $\tsr A\leq n$, there is a net  $(\bs u_\lambda)\in U_n(A)$ 
converging to $\bs x$.   Since $A$ is a topological algebra,
$v_\lambda:= \bs u_\lambda\cdot\bs a +yb$ tends to ${\bs 1}$. 
The openness of the set of units of $A$ now implies that  
$v_\lambda\in A^{-1}$ whenever $\lambda$ is large. We fix some of these $\lambda$.
  If $\bs u_\lambda=(u_1,\dots,u_n)$, then the ideal
$I_A(u_1,\dots,u_n)$ coincides with $A$. Hence there is $\bs y_\lambda\in A^n$
such that $y= \bs u_\lambda\cdot \bs y_\lambda$. Thus
$$v_\lambda= \bs u_\lambda \cdot( \bs a+ \bs y_\lambda \, b)\in A^{-1}.$$
Since $\tsr A=n$, we may approximate $\bs y_\lambda$ by $\bs w_\lambda\in U_n(A)$.
Hence $\bs u_\lambda \cdot( \bs a+ \bs w_\lambda \, b)\in A^{-1}$ whenever 
$\bs w_\lambda$ is sufficiently close to  $\bs y_\lambda$.  We conclude  that
$\bs a+\bs w_\lambda \,b\in U_n(A)$ and so $\usr A\leq n$.
\end{proof}

 The preceding result shows that  in case of a $Q$-algebra $A$,  $\tsr A=1$ is
a sufficient condition for   $\usr A=1$.

\section{The small-norm and the norm-one  ranks}

The following two concepts are  briefly mentioned in \cite{bad}.
\begin{definition}
Let $A=(A,||\cdot||)$ be a normed algebra.
\begin{enumerate}
\item[(1)]  $A$ is said to have the {\it norm-one  rank $n$} (denoted by $\nor A$) 
if $n$ is the smallest integer (or infinity) such that for every 
$(\bs f, g)\in U_{n+1}(A)$ there is  $\bs c=(c_1,\dots,c_n)\in A^n$ such that $||c_j||=1$
and
$$ \bs f+\bs c\, g\in U_n(A).$$

\item[(2)]  $A$ is said to have the {\it small-norm  rank $n$} (denoted by $\snr A$) 
if  $n$ is the smallest integer (or infinity) such that for every
$\e>0$ and every  $(\bs f,g)\in U_{n+1}(A)$ there is $\bs a=(a_1,\dots,a_n)
\in A^n$ such that $||a_j||<\e$ and
$$\bs f+\bs a\, g\in U_n(A).$$
\end{enumerate}

\end{definition}
$\bullet$~ I don't know whether these ranks have the stabilizing property.

Let $S_A=\{a\in A: ||a||=1\}$ be the unit sphere in $A$. 
The following relations now hold between the different  ranks. The striking point is
that the norm-one rank is bigger than the topological stable rank. This result is due to 
Badea \cite{bad}. We re-present here for the reader's convenience the simple proof.

\begin{proposition}[Badea] \label{nsr-tsr}
Let $A=(A,\n\,\cdot\n\,)$ be  normed $cQ$-algebra and $||\bs a||:=\sum_{j=1}^n \n\,a_j\,\n$\,, 
$\bs a\in A^n$. Then 
$$\bsr A\leq \usr A\leq \tsr A\leq \snr A\leq \nor\, A.$$
\end{proposition}
\begin{proof}
The first two inequalities are dealt with  in Proposition \ref{usr-bsr}. To show 
$\tsr A\leq  \snr A\leq\nor A$, we will use Theorem \ref{badi}.
So suppose  that $n:=\nor A<\infty$. Let $(\bs a, a_{n+1})\in U_{n+1}(A)$. Then, for every
$k\in \N^*$, $(\bs a, (1/k)a_{n+1})\in U_{n+1}(A)$. By hypothesis, there is $\bs d_k\in A^n\inter (S_A)^n$ (depending on $k$), such that 
$$\bs a+ \bs d_k \frac{a_{n+1}}{k}\in U_n(A).$$
Now given $\e>0$, choose $k=k(\e)$ so big that 
$$\max\left\{\frac{1}{\e}, \frac{\n\,a_{n+1}\,\n}{\e}\, n\right\} <k(\e).
$$ 
Let 
$\bs x:=\bs d_{k(\e)} / k(\e)$. Then $\n\, x_j\,\n\leq \e$ for $j=1,\dots,n$ and
$$\bs v:=\bs a+\bs x\, a_{n+1}\in U_{n}(A).$$
Thus $\snr A\leq n$. Moreover,  since
$||\bs v-\bs a||<\e$, we conclude from Theorem \ref{badi}, that $\tsr A\leq \snr A$.
\end{proof}

\begin{proposition}\label{snsr-nosr}
Let $A$ be  normed $cQ$-algebra.  Then 
$$\bsr A\leq \usr A\leq \aur A\leq \snr A\leq \nor A .$$
\end{proposition}
\begin{proof}
In view of Theorem \ref{nsr-tsr} it only remains to show that $\aur A\leq \snr A$.
Since $A^{-1}$ is open, we may chose $\delta>0$  so that for all $a\in A$, 
$||a-\bs 1||<\delta$ implies $a\in A^{-1}$.
Suppose  now that  $n:={\rm snsr}\, A<\infty$. Let $(\bs f,g)\in U_{n+1}(A)$ and put 
$\bs e:=(\bs 1,\dots,\bs 1)$.  Then $(\bs f -\bs e\, g, g)\in U_{n+1}(A)$. Given
$0<\e<\delta$,  there is, by assumption, $\bs x=(x_1,\dots,x_n)\in A^n$ with $||x_j||\leq \e$, such that
$$(\bs f-\bs e\, g)+ \bs x \, g\in U_n(A).$$
Hence $\bs f+ (\bs x-\bs e)\, g\in U_n(A)$.  But $a_j:=\bs 1-x_j\in A^{-1}$, because
$||a_j-\bs 1||= ||x_j||<\e<\delta$. Hence $\aur A\leq n$.
\end{proof}

Our main goal in this subsection is to determine the norm-one  rank of $C(X,\K)$. To this end,
we need  a refinement of Theorem \ref{badi} (in case of the algebra $A=C(X,\K)$). 
This refinement will say  that in the equation $\bs f+\bs y g\in U_n(C(X,\K))$, $n=\tsr C(X,\K)$, 
 we can actually  choose 
$\bs y=(y_1,\dots,y_n)$ in such a way that all its components $y_j$ have norm as small as we wish
(in Badea's result we had $||y_j g||_\infty<\e$).

\begin{proposition}\label{smallnorm}
Let $X$ be a compact Hausdorff space. Then
$$\snr C(X,\K)=\tsr C(X,\K).$$
\end{proposition}

\begin{proof}
In view of Proposition \ref{nsr-tsr}, it remains to show that $\snr C(X,\K)\leq \tsr C(X,\K)$.
So let  $n:=\tsr C(X,\K)<\infty$ and fix $(\bs f,g)\in U_{n+1}(C(X,\K))$. 

{\bf Case 1} $Z(g)=\emp$. Then $g$ is invertible and $(g^{-1}\,\bs f, 1)\in U_{n+1}(C(X,\K))$.
By Theorem \ref{badi},   for every $\e>0$, there is 
$\bs y=(y_1,\dots, y_n)\in C(X,\K^n)$, $||y_j\cdot 1||_\infty\leq \e$, such that 
$$g^{-1}\,\bs f + \bs y \cdot 1\in U_n(C(X,\K)).$$
 Hence $\bs f +\bs y \, g\in U_n(C(X,K))$.

{\bf Case 2}  $Z(g)\not=\emp$.
Choose an open   neighborhood
$U$ of $Z(g)$ such that $\bs f\not=\bs 0$ on $U$.  Let $V,W$ be two open sets  satisfying 
 $Z(g)\ss W\ss \ov W\ss  V\ss \ov V\ss U$. Since $X$ is normal, there is $\phi\in C(X,[0,1])$ which
$$\mbox{$\phi\equiv 0$  on $\ov V$ and $\phi=1$ on $X\setminus U$}.$$
Then $\ov V\ss Z(\phi)\ss U$. We deduce that $(\bs f,\phi)\in U_{n+1}(C(X,\K))$. Let $\e>0$ and
$$\delta:=\min\{|g(x)|: x\in X\setminus W\}.$$
Note that $\delta>0$.
Since, by assumption, $\tsr C(X,\K)=n$, we may use Theorem \ref{badi} to get a
function $\bs h=(h_1,\dots, h_n)\in C(X,\K^n)$ with 
$$\mbox{$\bs u:=\bs f+\bs h \phi\in U_n(C(X,\K))$ and $||h_j\phi||_\infty\leq \e \delta $}.$$
Now we define a function $\bs a=(a_1,\dots,a_n)$ by
$$a_j=\begin{cases}\dis \frac{1}{g}(u_j-f_j) & \text{on $X\setminus W$}\\
                                       0 & \text{ on $V$}
\end{cases}
$$
Since $\bs u=\bs f$ on $V\supseteq \ov W$, we conclude that $\bs a$ is well-defined and hence continuous.  Moreover,
$$|a_j|\leq \begin{cases}     \frac{1}{\delta}\; \e\delta=\e&\text{on $X\setminus W$}\\
                                        0& \text{on $V$}.
\end{cases}
$$
Thus $||a_j||_\infty\leq \e$. Finally
$$ \bs f+\bs a \,g= \begin{cases} \bs f+(\bs u-\bs f)=\bs u &\text{on $X\setminus W$}\\
                                        \bs f+ \bs 0= \bs u & \text{on $V$}.
\end{cases}
$$
In other words,  $\bs f+\bs a\,g=\bs u\in U_n(C(X,\K))$.
\end{proof}

\begin{theorem}\label{nsr-cx}
Let $X$ be a compact Hausdorff space. Then 
$$\bsr C(X,\K)=\tsr C(X,\K)=\nor C(X,\K).$$
\end{theorem}
\begin{proof}
By Vasershtein's result \cite{va}, we already have $\bsr C(X,\K)=\tsr C(X,\K)$. In view
of Proposition \ref{nsr-tsr}, it suffices to show that $\nor C(X,\K)\leq \tsr C(X,\K)$. 
Let $A= C(X,\K)$ and $n:=\tsr A$. 

{\bf Case 1} $n=1$. Let $(f,g)\in U_2(A)$.  First suppose that $Z(g)=\emp$. Since $\tsr A=1$,
there is $u\in A^{-1}$ such that $||g^{-1}f-u||_\infty\leq 1/2$. Now 
$$\mbox{$\dis g^{-1} f + \frac{u}{|u|}\not= 0$ on $X$},$$
because 
$$g^{-1} f + \frac{u}{|u|}= \big(g^{-1}f -u\big) + u\big(1+ \frac{1}{|u|}\big)=
\big(g^{-1}f -u\big) + \frac{u}{|u|}\;(1+|u|),$$
and the second summand has modulus strictly bigger than $1$. Hence
$$f+ \frac{u}{|u|}\, g\in U_1(A).$$

If $Z(g)\not=\emp$, we use Proposition \ref{smallnorm} to conclude that  there is $a\in A$
with $u:=f+ag\in U_1(A)$ and $||a||_\infty< 1/2$.   Approximating $a$ by an invertible function
we may assume that $a$ already is invertible. Since $f\not=0$ on $Z(g)$, say $|f|>\delta>0$ on 
$Z(g)$, we may choose two open sets $U$ and $V$ such that 
$$Z(g)\ss U\ss \ov U\ss V\ss \ov V\ss \{x\in X: |g|<\delta/2\}\inter \{x\in X: |f(x)| > \delta\}.$$
Let $x_0\in U$. We will construct a function $\phi\in A$ such that
$$\mbox{$|(a \phi)(x_0)|=1$ and $||a\phi||_\infty\leq 1$}$$
and  $f+(a\phi)g\not =0$ on $X$. To this end, let $\psi\in C(X,[0,1])$ satisfy
$$\mbox{$\psi\equiv 0$ on $X\setminus V$ and $\psi=1$ on $U$}$$
and let  $\phi$ be defined by
$$\phi= \frac{1}{a}\psi+ (1-\psi).$$
Then $\phi$ does the job. In fact,

$\bullet$~~ $(a\phi)(x_0)= \psi (x_0)+ a(x_0) \cdot 0=1$;

$\bullet$~~ $|a\phi|\leq  \psi+ |a|(1-\psi) \leq \psi+(1-\psi)=1$;

$ \bullet$~~ $|f+(a\phi) g|  = |f+a g| =|u| >0$ on $X\setminus V$ and

$ \bullet$~~ $|f+(a\phi) g|   \geq |f|-|a\phi|\, |g|\geq  \delta- 1\cdot |g|  \geq \delta/2>0$ on $V$.

We conclude that $\nor A=1$. So the case $n=1$ is settled completely.
\\

{\bf Case 2} $\tsr A=n<\infty$.

 For $\bs f=(f_1,\dots, f_n)\in A^n$, set  $||\bs f||=\sqrt{\sum_{j=1}^n||f_j||_\infty^2}$ and 
 $|\bs f|:=\sqrt{\sum_{j=1}^n |f_j|^2}$. Note that $|\bs f|\leq ||\bs f||$.

Let $(\bs f,g)\in U_{n+1}(A)$. 
We first assume that $Z(g)\not=\emp$.
By Theorem \ref{smallnorm}, there is $\bs y=(y_1,\dots,y_n)\in A^n$
with $\bs u:=\bs f+\bs y g\in U_n(A)$ and $|y_j| \leq 1/2$.
Since $\bs f\not=\bs 0$ on $Z(g)$, say $|\bs f|>\delta>0$ on 
$Z(g)$, we may choose two open sets $U$ and $V$ such that 
$$Z(g)\ss U\ss \ov U\ss V\ss \ov V\ss \{x\in X: |g|<\delta/(2\sqrt n)\}
\inter \{x\in X: |\bs f(x)| > \delta\}.$$
Fix $x_0\in Z(g)$. As above, let $\psi\in C(X,[0,1])$ satisfy
$$\mbox{$\psi\equiv 0$ on $X\setminus V$ and $\psi=1$ on $U$}.$$
For $j=1,\dots,n$,  let  $v_j$ be defined by
$$v_j= \psi+y_j (1-\psi),$$
and put $\bs v=(v_1,\dots,v_n)$. We claim that 
$$\mbox{$\bs f+ \bs v g\in U_n(A)$ and  $||v_j||_\infty=1$}.$$
In fact, 

$\bullet$~~ $ |v_j|\leq \psi +(1/2) (1-\psi)\leq 1$;

$\bullet$~~ $|v_j(x_0)|= \psi(x_0)=1$; hence $||v_j||_\infty=1$;

$\bullet$~~ $|\bs f+\bs v g| = |\bs f+ \bs y \, g|=|\bs u|>0$ on $X\setminus V$;

$\bullet$~~ $|\bs f+\bs v g| \geq |\bs f|- |g|\, |\bs v|\geq \delta  - \sqrt n  \delta/(2\sqrt n)=\delta/2$ on $V$.\\

Suppose now that $Z(g)=\emp$ and let $(\bs f,g)\in U_{n+1}(A)$, $n\geq 2$ (the case $n=1$ was
done in the preceding paragraph). Then
$(g^{-1}\,\bs f, 1)\in U_{n+1}(A)$ and it suffices to prove the existence
of $\bs v=(v_1,\dots,v_n)\in A^n$ such that $||v_j||_\infty=1$ and
$$g^{-1}\,\bs f+ \bs v\in U_n(A).$$
Let $\bs F:=g^{-1}\,\bs f$ and denote the coordinates of $F$ by $F_j$.
 Since $\tsr A=n$, there is $\bs u=(u_1,\dots,u_n)\in U_n(A)$ such that
$$||\bs F-\bs u||< 1/2.$$

We shall proceed inductively, with respect to the length of invertible subtuples {of $\bs u$,}
and will frequently use the following type of estimates.
Let  $\widetilde{\bs u}:=(u_1,\dots, u_m)\in U_m(A)$ and 
$$\bs v=(v_1,\dots, v_m):=\left(\frac{u_1}{|\widetilde{\bs u}|},\dots, 
\frac{u_m}{|\widetilde{\bs u}|}\right).$$
The hypothesis $\widetilde{\bs u} \in U_m(A)$ (or equivalently
$|\widetilde{\bs u}|\geq \delta>0$ on $X$) implies that $\bs v\in A^m$
 and each coordinate  of $\bs v$ has norm less than 1 (may be strict).  Moreover,
 if ${\widetilde{\bs F}=(F_1,\dots, F_m)}$,  then
\begin{equation}\label{m-tuple}
\bs {\widetilde F}+ \bs v\in U_m(A),
\end{equation}
 because
\begin{eqnarray*}
|\bs {\widetilde F}+\bs v| &=&|(\bs {\widetilde F}-\widetilde{\bs u})+
 (\widetilde{\bs u}+\bs v)|\\
 &\geq& |\widetilde{\bs u}+\bs v|-|\bs {\widetilde F}-\widetilde{\bs u}|\\
 &=&(1+|\widetilde{\bs u}|) -|\bs {\widetilde F}-\widetilde{\bs u}|\\
 &\geq& 1 -||\bs F-\bs u||\geq 1/2.
\end{eqnarray*}

$\bullet$~~  If $u_1\in A^{-1}=U_1(A)$ then, by the paragraph above for $m=1$, we see that 
$F_1+ u_1/|u_1|\in A^{-1}$. (Note that $|F_1-u_1|<1/2$). Hence
$$\left(F_1+ \frac{u_1}{|u_1|}, F_2 + 1,\dots, F_n+1\right)\in U_{n}(A).$$

$\bullet$~~  If $\bs u_{1,2}:=(u_1,u_2)\in U_2(A)$, but neither $u_1$ nor $u_2$ is in $U_1(A)$, then
there are $x_j\in X$ such that $u_j(x_j)=0$, $(j=1,2)$. Hence,  the coordinates of  
$$\bs v_{1,2}:=\left(\frac {u_1}{\sqrt{|u_1|^2+|u_2|^2}}, \frac {u_2}{\sqrt{|u_1|^2+|u_2|^2}}\right),
$$
have norm 1. Moreover, by \zit{m-tuple}.
$$\bs H_{1,2}:=\bs F_{1,2}+\bs v_{1,2}:=(F_1+v_1,F_2 +v_2)\in U_2(A),$$
and so
$$(F_1+v_1, F_2+v_2, F_3+1,\dots, F_n+1)\in U_n(A).$$

$\bullet$~~ If $\bs  u_{1,2,3}:=(u_1,u_2,u_3)\in U_3(A)$,  but neither $(u_1,u_2), (u_1,u_3)$
nor $(u_2,u_3)$ in $U_2(A)$, then  there are $x_{1,2}, x_{1,3}, x_{2,3}\in X$ such that 
$u_i(x_{1,2})=0$, $(i=1,2)$,   $u_i(x_{1,3})=0$, $(i=1,3)$, and $u_i(x_{2,3})=0$, $(i=2,3)$.
Hence,  the coordinates of  
$$\bs v_{1,2,3}:=\left(\frac {u_1}{\sqrt{|u_1|^2+|u_2|^2+|u_3|^2}}, 
\frac {u_2}{\sqrt{|u_1|^2+|u_2|^2+|u_3|^2}},
\frac {u_3}{\sqrt{|u_1|^2+|u_2|^2+|u_3|^2}}\right),
$$
have norm 1. Moreover, by \zit{m-tuple}.
$$\bs H_{1,2,3}:=\bs F_{1,2,3}+\bs v_{1,2,3}:=(F_1+v_1,F_2 +v_2, F_3+v_3)\in U_3(A),$$
and so
$$(F_1+v_1,\dots, F_3+v_3, F_4+1,\dots, F_n+1)\in U_n(A).$$

Now we proceed  inductively up to the $n$-th step. Since $\bs u\in U_n(A)$,
we may assume (by the induction hypothesis), that no subtuple of order $n-1$ is
invertible. Then we may choose $x_j\in \Inter_{k\not=j} Z(u_k)\not=\emp$, $j=1,\dots, n$.
Consequently,  the coordinates of 
$$\bs v:= \frac{\bs u}{|\bs u|}=\left(\frac{u_1}{|\bs u|}, \dots, \frac{u_n}{|\bs u|}\right)$$ 
each have norm one.  Since $\bs F+\bs v\in U_n(A)$ (by \zit{m-tuple}), we are done.

{\bf Case 3} $\tsr A=\infty$. By Theorem \ref{nsr-tsr}, $\nor A$ cannot be finite in that case.
Hence we deduce from all the three cases above that 
$\nor A\leq \tsr A\leq \nor A$, and so  we have equality of all the three stable ranks for $C(X,\K)$.
\end{proof}

A combination of the previous results now yields:

\begin{corollary}
Let $X$ be  a compact Hausdorff space and $A=C(X,\K)$. Then 
$$\bsr A=\usr A=\aur A=\tsr A= \snr A=\nor A.$$
\end{corollary}

Recall that in the context of the algebras $C(X,\K)$, the original question by Rieffel
reads as follows:

$\bullet$~  Given $(\bs f,g)\in U_{n+1}(C(X,\K))$, when does there exist $\bs u=(u_1,\dots,u_n)\in C(X,\K^n)$
with $\bs f+\bs u\, g\in U_n(C(X,\K))$ such that all the components $u_j$ of $\bs u$
 have modulus one? It remains unanswered.

\section{General uniform algebras}
Given a commutative unital normed algebra $A$, 
let us call an $(n+1)$-tuple $(\bs f, g)\in U_{n+1}(A)$ {\it norm-one reducible},
if there exists $\bs c=(c_1\dots,c_n)\in A^n$ such that $||c_j||=1$ and 
$\bs f+\bs c\, g\in U_n(A)$.  In the previous section  we have shown that
in $C(X,\K)$ every invertible $(n+1)$-tuple is norm-one reducible, provided $\tsr C(X,\K)=n$.
Using those ideas, we give a sufficient condition on tuples to be norm-one reducible
in an arbitrary uniform algebra. The proof is based on the theory of (weak) peak-points and
 the following function theoretic Lemma  from \cite[p. 491]{gm1}. Recall that a point $x\in X$
 is a {\it weak peak point} for a uniformly closed subalgebra $A$ of $C(X,\C)$   if
 $\{x\}$ is an intersection of peak-sets (these are closed subsets $E$ of $X$ for which there
 exists $f\in A$ such that $f(\xi)=1$  if $\xi\in E$ and $|f(\xi)|<1$ if $\xi\in X\setminus E$).

\begin{lemma}\label{mob} {\sl Let $0<\eta<1$ and $0<\e<1$. Then there exists an
automorphism $L$ of the unit disk with fixed points $-1$ and $1$,
 and a positive zero $a$ such that the image of
$\{z\in\ov {\D}: |z-1|>\eta\}$ under  $L$ is contained in $\{w\in\ov {\D}: |w+1|<\e\}$.}
\end{lemma}

\begin{proposition}\label{nsr-uni}
Let $A$ be a uniform algebra. We view $A$ as a uniformly closed subalgebra
of $C(X,\C)$, where $X=M(A)$. Suppose that $n:=\snr A<\infty$ and  
let $(\bs f, g)\in U_{n+1}(A)$. 
 Then $(\bs f,g)$  is norm-one reducible if $Z(g)$ meets the Shilov boundary.
\end{proposition}
\begin{proof}
Recall that by  Proposition \ref{snsr-nosr} that ${\rm snsr}\, A\leq \nor A$. 
If $f_j\equiv 0$ on $X$ for every $j$, then $(0+1\cdot g, \dots, 0+1\cdot g)\in U_n(A)$
is a solution to our norm-controlled reducibility. So we may assume that not all the $f_j$
are the zero functions. If $g\equiv 0$, then $\bs f\in U_n(A)$ and we take
$\bs f+ \bs e\cdot g$ as a solution, where $\bs e=(\bs 1,\dots,\bs 1)$.

 Let $E=\partial A$ be the Shilov boundary of $A$.
By our assumption,  $Z(g)\inter E\not=\emp$. Since $\snr A=n$,
 there is $\bs y=(y_1,\dots,y_n)\in A^n$
with $\bs u:=\bs f+\bs y \,g\in U_n(A)$ and $||y_j||< 1/2$.  
Let $x_0\in Z(g)\inter E$. 
Since $\bs f\not=\bs 0$ on $Z(g)$, say $|\bs f|>\delta>0$ on 
$Z(g)$, we may choose two open sets $U$ and $V$ such that 
$$x_0\ss U\ss \ov U\ss V\ss \ov V\ss \{x\in X: |g|<\delta/(2\sqrt n)\}
\inter \{x\in X: |\bs f(x)| > \delta\}.$$
 Because $E$ is the closure of the set of weak-peak points \cite{gam}, $U\inter E$ contains such a point $x_1$. Hence, there is a peak-set $S$ such that $x_1\in S\ss U$. 
Choose a peak function $q\in A$ associated with $S$.   Let  $m\in \N$, $m\geq 2$, be
 so big that on $X\setminus V$ the function $\Phi:=[(1+q)/2]^m$ satisfies
 $$|\Phi|\leq 1/2.$$
 
  Let $\eta>0$ be such that 
 $$\{z\in \D: |z|\leq 1/2\}\ss \{z\in\ov {\D}: |z-1|>\eta\},$$
 and put
 $$\e:= \frac{\delta'}{4\sqrt n ||g||_\infty},$$
 where $\delta':=\min_X |\bs u|$. Consider the M\"obius transform
of Lemma \ref{mob}with $L(1)=1$, $L(-1)=-1$,
$$
L(\{z\in\ov {\D}: |z-1|>\eta\})\ss \{w\in\ov {\D}: |w+1|<\e\}.
$$
Then $\psi:=(1+L\circ\Phi)/2$  again is a peak function in $A$ associated with $S$
(note that the membership in $A$ is given by the functional calculus: $\sigma (\Phi)\ss\ov \D$
and $L$ holomorphic in  a neighborhood of $\ov{\D}$). Due to the choice of our parameters,
 $\psi\sim 0$ on $X\setminus V$; more precisely, 
$$|\psi|\leq \frac{\delta'}{8 \sqrt n ||g||_\infty}.$$ 
For $j=1,\dots,n$,  let  $v_j$ be defined by
$$v_j= \psi^2 +y_j(1-\psi)^2=
\left(\frac{1+L\circ \Phi}{2}\right)^2+y_j\, \left(\frac{1-L\circ\Phi}{2}\right)^2,$$
and put $\bs v=(v_1,\dots,v_n)$. Then $\bs v\in A^n$. We claim that 
$$\mbox{$\bs f+ \bs v g\in U_n(A)$ and  $||v_j||_\infty=1$}.$$

In fact, since $x_1\in S\inter E$,  $|v_j(x_1)|=1.$ Moreover if $p:=L\circ\Phi$,
\begin{eqnarray*}
|v_j|&\leq &\left|\frac{1+p}{2}\right|^2+|y_j|\, \left|\frac{1-p}{2}\right|^2\\
        &\leq & \left|\frac{1+p}{2}\right|^2+1 \cdot \, \left|\frac{1-p}{2}\right|^2\\
        &\leq & \frac{1}{4} \Big((1+|p|^2 +2 {\rm Re}\, p) + (1+|p|^2 -2 {\rm Re}\, p)\Bigr)\\
        &\leq & \frac{1}{4} \cdot 4=1
        \end{eqnarray*}

\noindent Moreover

$\bullet$~~ $|\bs f+\bs v g| \geq |\bs f|-|\bs v|\, |g|\geq \delta  - \sqrt n \; \delta/(2\sqrt n)=\delta/2$
 on $V$ and

$\bullet$~~ $|\bs f+\bs v g| \geq |\bs f +\bs y g|-|\bs v-\bs y|\, |g| =  |\bs u|-|\bs v-\bs y|\, |g| $ on $X\setminus V$.\\

\noindent But $v_j-y_j=\psi^2 +y_j(1+\psi^2-2\psi)-y_j=\psi^2+y_j\psi^2-2\psi y_j=\psi(\psi+y_j\psi-2y_j)$.
Hence, on $X\setminus V$, 
$$|v_j-y_j|\leq 4|\psi| \leq 4\frac{\delta'}{8\,||g||_\infty\sqrt n}.$$
Consequently, on $X\setminus V$, 
$$|\bs f+\bs v\,g|\geq \delta'-||g||_\infty \; \frac{\delta'}{2 ||g||_\infty \sqrt n} \sqrt n =\delta'/2>0.$$

\end{proof}


\begin{thebibliography}{99}
 
  \bibitem{bad}  C. Badea.
  The stable rank of topological algebras and a problem of R.G. Swan,
 J. Funct. Anal. 160 (1998), 42--78.
 
 \bibitem{ccm} J. J. Carmona, J. Cufi, P. Menal.
 On the unit-1-stable rank of rings of analytic functions.
 Publ. Mat. 36 (1992), 439-447.
 
 \bibitem{gam} {T.W. Gamelin}.
 \emph{Uniform algebras}, 
Chelsea, New York, 1984.
 
   \bibitem{gm1} P. Gorkin, R. Mortini.
Asymptotic interpolating sequences in uniform algebras.  
 J. London Math. Soc.  67  (2003),  481-- 498.
 
 
\bibitem{jmw} P.W. Jones, D. Marshall, T.H. Wolff, 
Stable rank of the disc algebra,
Proc. Amer. Math. Soc. 96 (1986), 603-604.

\bibitem{mvdk} B. Mirzaii, W. van der Kallen,
Homology stability for unitray groups,
Documenta Math. 7 (2002), 143--166.
 
 \bibitem{moru92}   R. Mortini, R. Rupp.
Totally reducible elements in rings of analytic functions, 
 Comm. Algebra  20, (1992), 1705--1713.
 
  \bibitem{ri} M. Rieffel.
   Dimension  and stable rank in the $K$-theory of $C^*$-algebras,
Proc. London Math. Soc. 46 (1983), 301--333.

\bibitem{ro} A.G. Robertson.
Stable range in $C^*$-algebras, 
Math. Proc. Cambridge Phil. Soc. 87 (1980), 413-418.


 \bibitem{va} L. Vasershtein.
Stable rank of rings and dimensionality of topological spaces,
Funct. Anal. Appl.  5 (1971), 102--110;
translation from Funkts. Anal. Prilozh. 5 (1971), No.2, 17--27.

\end{thebibliography}
 \end{document}